\documentclass[10pt]{amsart}
\usepackage{amscd,amsmath,amssymb,amsthm,enumerate,graphicx}
\newtheorem{definition}[equation]{Definition}
\newtheorem{lemma}[equation]{Lemma}
\newtheorem{proposition}[equation]{Proposition}
\newtheorem{theorem}[equation]{Theorem}
\newtheorem{corollary}[equation]{Corollary}
\newtheorem{remark}[equation]{Remark}

\newcommand\lemmaref[1]{Lemma~\ref{#1}}
\newcommand\propositionref[1]{Proposition~\ref{#1}}
\newcommand\theoremref[1]{Theorem~\ref{#1}}
\newcommand\corollaryref[1]{Corollary~\ref{#1}}

\title{Classification of Klein Four Symmetric Pairs of Holomorphic Type for $\mathrm{E}_{6(-14)}$}
\author{Haian HE}
\date{}
\address{Department of Mathematics,
College of Sciences, Shanghai University,
No.\ 99 Shangda Road, Baoshan District,
Shanghai, China P.\ R.\ 200444}
\email{hebe.hsinchu@yahoo.com.tw}

\begin{document}
\begin{abstract}
The author classifies Klein four symmetric pairs of holomorphic type for non-compact Lie group $\mathrm{E}_{6(-14)}$, which gives a class of pairs $(G,G')$ of real reductive Lie group $G$ and its reductive subgroup $G'$ such that there exist irreducible unitary representations $\pi$ of $G$, which are admissible upon restriction to $G'$.
\end{abstract}
\maketitle
\section{Introduction}
\subsection{Notations}
Denote by $\mathbb{R}$ and $\mathbb{C}$ the real field and the complex field respectively. For any positive integer $n$, denote by $I_n$ the $n\times n$ identity matrix. For a Lie group $G$ with its Lie algebra $\mathfrak{g}$, denote by $\mathrm{Aut}G$ and $\mathrm{Aut}\mathfrak{g}$ the automorphism group of $G$ and $\mathfrak{g}$ respectively. Also, denote by $\mathrm{Int}\mathfrak{g}$ the subgroup of $\mathrm{Aut}\mathfrak{g}$, which contains all the inner automorphisms of $\mathfrak{g}$. For $f\in\mathrm{Aut}G$, write $G^f:=\{g\in G\mid f(g)=g\}$. Similarly, for $f\in\mathrm{Aut}\mathfrak{g}$, write $\mathfrak{g}^f:=\{X\in\mathfrak{g}\mid f(X)=X\}$ and $\mathfrak{g}^{-f}:=\{X\in\mathfrak{g}\mid f(X)=-X\}$. Moreover, if two elements $g_1$ and $g_2$ in the group $G$ are conjugate by an element in the subgroup $H$ of $G$, i.e., $g_2=h^{-1}g_1h$ for some $h\in H$, then write $g_1\sim_Hg_2$; else, write $g_1\nsim_Hg_2$. For $g\in G$, denote by $[g]$ the conjugacy class of $g$ in $G$. If $g_1,g_2,\cdots,g_n$ are $n$ elements in a group $G$, then denote by $\langle g_1,g_2,\cdots,g_n\rangle$ the subgroup of $G$ generated by $g_1,g_2,\cdots,g_n$.
\subsection{Outlines}
Riemannian symmetric pairs were classified by \'{E}lie Joseph CARTAN in \cite{C1} and \cite{C2}, and the general semisimple symmetric pairs were classified by Marcel BERGER in \cite{B}. Jingsong HUANG and Jun YU studied semisimple symmetric spaces from a different point of view in \cite{HY}; that is, by determining the Klein four subgroups in the automorphism groups of compact Lie algebras. As an extension of the work of \cite{HY}, Jun YU classified all the elementary abelian 2-groups in the automorphism groups of compact Lie algebras, where groups of rank 2 are just Klein four subgroups. However, the classification of Klein four subgroups in the automorphism groups of non-compact simple Lie algebras is not clear. In this acticle, the author will study the Klein four symmetric pairs for the connected simple Lie group $\mathrm{E}_{6(-14)}:=\mathrm{Int}\mathfrak{e}_{6(-14)}$; namely, to find some pairs $(\mathrm{E}_{6(-14)},G')$ with $G'$ the reductive subgroup of the fixed points under the action of some Klein four subgroup of the automorphism group of $\mathrm{E}_{6(-14)}$.

The motivation to study the Klein four symmetric pairs for $\mathrm{E}_{6(-14)}$ comes from the branching problem raised by Toshiyuki KOBAYASHI in \cite[Problem 5.6]{Ko5}. Concretely, one may want to classify the pairs $(G,G')$ of real reductive Lie groups satisfying the following condition: there exist an infinite dimensional irreducible unitary representation $\pi$ of $G$ and an irreducible unitary representation $\tau$ of $G'$ such that\[0<\dim_\mathbb{C}\mathrm{Hom}_{\mathfrak{g}',K'}(\tau_{K'},\pi_K|_{\mathfrak{g}'})<+\infty\]where $K$ is a maximal compact subgroup of $G$, $\mathfrak{g}$ is the complexified Lie algebra of $G$, $\pi_K|_{\mathfrak{g}'}$ is the restriction of underlying $(\mathfrak{g},K)$-module of $\pi$ to the complexified Lie algebra $\mathfrak{g}'$ of $G'$, $K'$ is a maximal compact subgroup of $G'$ such that $K'\subseteq G'\cap K$, and $\tau_{K'}$ is the underlying $(\mathfrak{g}',K')$-module of $\tau$. For symmetric pairs $(G,G')$, the problem was solved in \cite[Theorem 5.2]{KO}. Unfortunately, it is unknown for general cases, even for Klein four symmetric pairs.

Here is a sufficient condition. If there exist an infinite dimensional irreducible unitary representation $\pi$ of $G$ and an irreducible unitary representation $\tau$ of $G'$ such that\[0<\dim_\mathbb{C}\mathrm{Hom}_{G'}(\tau,\pi|_{G'})<+\infty\]where $\pi|_{G'}$ is the restriction of $\pi$ to $G'$. This is true when $G$ is a simple Lie group of Hermitian type and $(G,G')$ is a symmetric pair of holomorphic type, which the author will tell the details in this article. As a generalization of this result, if $G$ is a simple Lie group of Hermitian type and $(G,G')$ is a Klein four symmetric pair of holomorphic type, this is also true. More general, for a simple Lie group of Hermitian type $G$ and its reductive subgroup $G'$, if the Lie algebra of $G'$ contains the center of the Lie algebra of a maximal compact subgroup of $G$, then it holds. It is not easy to give a complete list of all such pairs $(G,G')$, but one may just consider Klein four symmetric pairs. Hence, in this article, the author focus on the case when $G=\mathrm{E}_{6(-14)}$.

The article is organized as follows. At the beginning, the author briefly introduces Klein four symmetric pairs. After that, the author explains how Klein four symmetric pairs of holomorphic type is related to representations of Lie groups. Finally, the author classifies Klein four symmetric pairs of holomorphic type for $\mathfrak{e}_{6(-14)}$, and applies the results to representation theory.
\section{Klein Four Symmetric Pairs}
\subsection{Finite groups in $\mathrm{Aut}\mathfrak{g}_0$}
Let $\mathfrak{g}_0$ be a real semisimple Lie algebra. For each element $\sigma\in\mathrm{Aut}\mathfrak{g}_0$, define $\Theta(\sigma):\mathrm{Aut}\mathfrak{g}_0\rightarrow\mathrm{Aut}\mathfrak{g}_0$ given by $f\mapsto\sigma f\sigma^{-1}$ whose differential is $\sigma$. In particular, if $\theta$ is a Cartan involution of $\mathfrak{g}_0$, then $\Theta(\theta)$ is a Cartan involution of $\mathrm{Aut}\mathfrak{g}_0$, and the centralizer $Z_{\mathrm{Aut}\mathfrak{g}_0}(\theta):=\{\sigma\in\mathrm{Aut}\mathfrak{g}_0\mid\sigma\theta=\theta\sigma\}$ is a maximal compact subgroup of $\mathrm{Aut}\mathfrak{g}_0$.

Conversely, Suppose that $K$ is a maximal compact subgroup of $\mathrm{Aut}\mathfrak{g}_0$. Take a Cartan involution $\theta'$ of $\mathfrak{g}_0$, and $Z_{\mathrm{Aut}\mathfrak{g}_0}(\theta')$ is a maximal compact subgroup of $\mathrm{Aut}\mathfrak{g}_0$. Hence, there exists an element $f\in\mathrm{Aut}\mathfrak{g}_0$ such that $K=f^{-1}Z_{\mathrm{Aut}\mathfrak{g}_0}(\theta')f=Z_{\mathrm{Aut}\mathfrak{g}_0}(f^{-1}\theta'f)$, and $\theta:=f^{-1}\theta'f$ is also a Cartan involution of $\mathfrak{g}_0$. Therefore, any maximal compact subgroup of $\mathrm{Aut}\mathfrak{g}_0$ is given by $Z_{\mathrm{Aut}\mathfrak{g}_0}(\theta)$ for some Cartan involution $\theta$ of $\mathfrak{g}_0$.
\begin{proposition}\label{2}
If $H$ is a finite subgroup of $\mathrm{Aut}\mathfrak{g}_0$, then there exists a Cartan involution $\theta$ of $\mathfrak{g}_0$, which centralizes $H$.
\end{proposition}
\begin{proof}
It is clear that the finite group $H$ is included in a maximal compact subgroup of $\mathrm{Aut}\mathfrak{g}_0$, so $H\subseteq Z_{\mathrm{Aut}\mathfrak{g}_0}(\theta)$ for some Cartan involution $\theta$ of $\mathfrak{g}_0$; namely, $\theta$ centralizes $H$.
\end{proof}
\subsection{Klein Four Subgroups}
From now on, the author focuses on real simple Lie groups $G$ and real simple Lie algebras $\mathfrak{g}_0$.
\begin{definition}\label{8}
Let $G$ (respectively, $\mathfrak{g}_0$) be a real simple Lie group (respectively, Lie algebra), and let $\sigma$ and $\tau$ be two non-identity involutive automorphisms of $G$ (respectively, $\mathfrak{g}_0$). Then the subgroup $G^{\{\sigma,\tau\}}:=G^\sigma\cap G^\tau$ (respectively, $\mathfrak{g}_0^{\{\sigma,\tau\}}:=\mathfrak{g}_0^\sigma\cap\mathfrak{g}_0^\tau$) is called a Klein four symmetric subgroup (respectively, subalgebra) of $G$ (respectively, $\mathfrak{g}_0$), and $(G,G^{\{\sigma,\tau\}})$ (respectively, $(\mathfrak{g}_0,\mathfrak{g}_0^{\{\sigma,\tau\}})$) is called a Klein four symmetric pair. Two Klein four symmetric pairs $(G_1,G'_1)$ (respectively, $(\mathfrak{g}_1,\mathfrak{g}'_1)$) and $(G_2,G'_2)$ (respectively, $(\mathfrak{g}_2,\mathfrak{g}'_2)$) are said to be isomorphic if there exists a Lie group (respectively, Lie algebra) isomorphism $f:G_1\rightarrow G_2$ (respectively, $f:\mathfrak{g}_1\rightarrow\mathfrak{g}_2$) such that $f(G'_1)=G'_2$ (respectively, $f(\mathfrak{g}'_1)=\mathfrak{g}'_2$).
\end{definition}
\begin{remark}\label{10}
The subgroup $\langle\sigma,\tau\rangle$ of $\mathrm{Aut}G$ or $\mathrm{Aut}\mathfrak{g}_0$ generated by the set $\{\sigma,\tau\}$ is a Klein four subgroup, and it is obvious that $G^{\{\sigma,\tau\}}=G^{\langle\sigma,\tau\rangle}:=\displaystyle{\bigcap_{\gamma\in\langle\sigma,\tau\rangle}}G^\gamma$ or $\mathfrak{g}_0^{\{\sigma,\tau\}}=\mathfrak{g}_0^{\langle\sigma,\tau\rangle}:=\displaystyle{\bigcap_{\gamma\in\langle\sigma,\tau\rangle}}\mathfrak{g}_0^\gamma$.
\end{remark}
\begin{remark}\label{11}
According to the classification of Klein four symmetric pair for compact exceptional Lie algebras $\mathfrak{u}_0$ \cite[Table 4]{HY}, the conjugacy classes of Klein four symmetric subalgebras in $\mathfrak{u}_0$ are in one-to-one correspondence to the conjugacy classes of Klein four subgroups in $\mathrm{Aut}\mathfrak{u}_0$. However, it is not true for classical Lie algebras. For example, there are two non-conjugate Klein four subgroups $\Gamma_1$ and $\Gamma_2$ in $\mathrm{Aut}\mathfrak{su}(4)$, such that $\mathfrak{su}(4)^{\Gamma_1}\cong2\mathfrak{so}(2)\cong2\mathfrak{sp}(1)\cong\mathfrak{su}(4)^{\Gamma_2}$ by \cite[Table 3]{HY}.
\end{remark}
\subsection{Klein Four Symmetric Pairs of Holomorphic Type}
Take a Cartan involution $\theta$ for a real simple Lie group $G$, which defines a maximal compact subgroup $K$. Let $\tau$ be an involutive automorphism of $G$, which commutes with $\theta$. Use the same letter $\tau$ to denote its differential, and then $\tau$ stabilizes $\mathfrak{k}_0$ and also the center $Z(\mathfrak{k}_0)=\mathbb{R}Z$. Because $\tau^2=1$, there are two possibilities: $\tau Z=Z$ or $\tau Z=-Z$. Recall \cite[Definition 1.4]{Ko4} that the symmetric pair $(G,G^\tau)$ or $(\mathfrak{g}_0,\mathfrak{g}_0^\tau)$ is said to be of holomorphic (respectively, anti-holomorphic) type if $\tau Z=Z$ (respectively, $\tau Z=-Z$), in which case $\tau$ may be said to be of holomorphic (respectively, anti-holomorphic) type for convenience.

According to the classification of the symmetric pairs of holomorphic type and anti-holomorphic type in \cite[Table 3.4.1 \& Table 3,4,2]{Ko4}, there does not exist a symmetric pair which is of both holomorphic type and anti-holomorphic type, so the author remarks that whether a symmetric pair is of holomorphic type or anti-holomorphic type does not depends on the choice of the Cartan involution which commute with the given involutive automorphism defining the symmetric pair.
\begin{definition}\label{16}
Suppose that $G$ (respectively, $\mathfrak{g}_0$) is a real simple Lie group (respectively, Lie algebra) of Hermitian type. Let $\tau$ and $\sigma$ be two involutive automorphisms of $G$ such that $\tau\sigma=\sigma\tau$. If both $\tau$ and $\sigma$ are of holomorphic type, then $(G,G^{\{\tau,\sigma\}})$ or $(\mathfrak{g}_0,\mathfrak{g}_0^{\{\tau,\sigma\}})$ is called a Klein four symmetric pair of holomorphic type.
\end{definition}
Unlike symmetric pairs of holomorphic type, the classification of Klein four symmetric pairs of holomorphic type is unknown. In the last section, the author will classify Klein four symmetric pairs of holomorphic type for $\mathrm{E}_{6(-14)}$ which is a simple Lie group of Hermitian type.
\subsection{Admissible Representations}
Suppose that $G$ is a real simple Lie group, and $G'$ is a reductive subgroup of $G$. Let $\pi$ be a unitary representation of $G$ on a Hilbert space. Recall the definition of $G'$-admissible representations given in \cite{Ko1} and \cite{Ko2}. The restriction $\pi|_{G'}$ from $G$ to $G'$ is called $G'$-admissible if the restriction $\pi|_{G'}$ splits into a discrete Hilbert direct sum of irreducible representations:\[\pi|_{G'}\cong\displaystyle{\widehat{\bigoplus}_{\tau\in\widehat{G'}}}n_\pi(\tau)\tau,\]with each multiplicity $n_\pi(\tau)$ a nonnegative integer, where $\widehat{G'}$ denotes the unitary dual of $G'$.

The point here is that there is no continuous spectrum in the branching law and that each multiplicity is finite. Here is a special but important example of $G'$-admissible restrictions, which was shown in \cite{HC1}, \cite{HC2}, and \cite{HC3}: If $K$ is a maximal compact subgroup of $G$, any irreducible unitary representation $\pi$ of $G$ is $K$-admissible when restricted to $K$.

Suppose that $G$ is a real simple Lie group of Hermitian type; that is, $G/K$ carries a structure of a Hermitian symmetric space where $K$ is a maximal compact subgroup of $G$. Equivalently, the center $Z(\mathfrak{k}_0)$ of Lie algebra $\mathfrak{k}_0$ of $K$ has dimension 1. The classification of simple Lie algebras $\mathfrak{g}_0$ of Hermitian type is given as follows:\[\mathfrak{su}(p,q),\mathfrak{sp}(n,\mathbb{R}),\mathfrak{so}(m,2)(n\neq2),\mathfrak{e}_{6(-14)},\mathfrak{e}_{7(-25)}.\]Such a Lie algebra $\mathfrak{g}_0$ satisfies the condition that a Cartan subalgebra $\mathfrak{h}_0$ of $\mathfrak{k}_0$ becomes a Cartan subalgebra of $\mathfrak{g}_0$. Moreover, there exists a characteristic element $Z\in Z(\mathfrak{k}_0)$ such that $\mathfrak{g}=\mathfrak{l}\oplus\mathfrak{p}_+\oplus\mathfrak{p}_-$ is a decomposition with respect to the eigenspaces of $Z$ on the complexified Lie algebra $\mathfrak{g}$ corresponding to the eigenvalue 0, $\sqrt{-1}$, and $-\sqrt{-1}$ respectively. Similarly, Remove the subscript, and then $\mathfrak{h}$ denotes the complexification of $\mathfrak{h}_0$. Choose a positive system $\Phi^+$ for $(\mathfrak{g},\mathfrak{h})$ and its simple system $\Delta$. Denote by $\Phi_\mathfrak{k}^+$ the set of the positive roots of $(\mathfrak{k},\mathfrak{h})$, and write $\Phi_\mathfrak{p}^+:=\Phi^+\setminus\Phi_\mathfrak{k}^+$. Set $\Delta_\mathfrak{k}:=\Phi_\mathfrak{k}^+\cap\Delta$ and $\Delta_\mathfrak{p}:=\Delta\setminus\Delta_\mathfrak{k}$. Obviously, $\Delta_\mathfrak{p}$ contains exactly 1 element because $Z(\mathfrak{k})$ has complex dimension 1.

Suppose that $V$ is a simple $(\mathfrak{g},K)$-module, and then set $V^{\mathfrak{p}_+}=\{v\in V\mid Yv=0\textrm{ for any }Y\in\mathfrak{p}^+\}$. Since $K$ normalizes $\mathfrak{p}_+$, $V^{\mathfrak{p}_+}$ is a $K$-submodule. Further, $V^{\mathfrak{p}_+}$ is either zero or an irreducible finite-dimensional representation of $K$. A $(\mathfrak{g},K)$-module $V$ is called a highest weight module if $V^{\mathfrak{p}_+}\neq\{0\}$.

Recall the definition of unitary highest weight representation. If $\pi$ is a unitary representation of $G$ on a Hilbert space $\mathcal{H}$, and $\mathcal{H}_K$ is the underlying $(\mathfrak{g},K)$-module, then $\pi$ is called a unitary highest weight representation if $\mathcal{H}_K^{\mathfrak{p}_+}\neq\{0\}$; namely, $\mathcal{H}_K$ is a highest weight $(\mathfrak{g},K)$-module.
\begin{proposition}\label{13}
Let $G$ be a real simple Lie group of Hermitian type, and $G'$ a reductive subgroup of $G$. Let $K$ be a maximal compact subgroup of $G$ such that $K':=G'\cap K$ is a maximal compact subgroup of $G'$. Suppose that the Lie algebra $\mathfrak{g}'$ of $G'$ contains the center $Z(\mathfrak{k}_0)$ of the Lie algebra $\mathfrak{k}_0$ of $K$. If $\pi$ is an irreducible unitary highest weight representation of $G$, the restriction $\pi\mid_{G'}$ is $G'$-admissible.
\end{proposition}
\begin{proof}
Let $\lambda$ be the highest weight of the underlying $(\mathfrak{g},K)$-module $V_K$ of $\pi$, and then each weight in $V_K$ is of the form $\lambda-\beta-n\alpha$ for some non-negative integer $n$, where $\beta$ is an integer linear combination of elements in $\Phi_\mathfrak{k}^+$ and $\alpha$ is the unique simple root in $\Delta_\mathfrak{p}$. It is obvious that the weight vectors with any fixed weight of $V_K$ form a finite-dimensional space. Take an element $Z\in Z(\mathfrak{k})$ such that $\alpha(Z)=1$, and the weight vectors with the weight $\lambda-\beta-n\alpha$ lie in the eigenspace of $Z$ with the eigenvalue $\lambda(Z)-n$; namely, the eigenvalue of $Z$ on each weight vector only depends on $n$. Because $\Phi_\mathfrak{k}^+$ is a finite set, each eigenspace is of finite-dimensional. Notice that $K'$ is compact, $V_K$ is decomposed as a direct sum of $K'$-representations. By assumption, $Z$ is in the complexified Lie algebra of $K'$, and hence the vectors in the same $K'$-type have the same eigenvalue of $Z$; namely, each $K'$-type lies in some eigenspace of $Z$. It follows that $V_K$ is $K'$-admissible, and thus the restriction $\pi|_{K'}$ is $K'$-admissible. By \cite[Theorem 1.2]{Ko2}, the restriction $\pi|_{G'}$ is $G'$-admissible.
\end{proof}
\begin{remark}\label{15}
\propositionref{13} is a generalization of the result \cite[Theorem 7.4(4)]{Ko3} which deals with the case when $(G,G')$ is a symmetric pair, but the proof for \cite[Theorem 7.4(4)]{Ko3} does not involve any property of symmetric pairs, so the method of the proof can also be applied to the generalized case when $(G,G')$ is a pair of reductive Lie groups. However, in the proof of \propositionref{13}, the author avoids constructing holomorphic sections of holomorphic vector bundles as in the proof for \cite[Theorem 7.4(4)]{Ko3}.
\end{remark}
\begin{corollary}\label{17}
If $(G,G^{\{\tau,\sigma\}})$ is a Klein four symmetric pair of holomorphic type, then any unitary highest weight representation $\pi$ of $G$ is $G'$-admissible.
\end{corollary}
\begin{proof}
Since $\tau\sigma=\sigma\tau$, by \propositionref{2}, there exists a maximal compact subgroup $K$ of $G$, which is fixed by both $\tau$ and $\sigma$. Moreover, both $\mathfrak{g}_0^\tau$ and $\mathfrak{g}_0^\sigma$ contains the center $Z(\mathfrak{k}_0)$ of $\mathfrak{k}_0$ because both $\tau$ and $\sigma$ are of holomorphic type. Thus, $Z(\mathfrak{k}_0)\subseteq\mathfrak{g}_0^{\{\tau,\sigma\}}$. The conclusion follows from \propositionref{13}.
\end{proof}
\section{Klein Four Symmetric pairs of holomorphic type for $\mathrm{E}_{6(-14)}$}
\subsection{Elementary abelian 2-subgroups in $\mathrm{Aut}\mathfrak{e}_{6(-78)}$}
Let $\mathfrak{g}=\mathfrak{e}_6$, the complex simple Lie algebra of type $\mathrm{E}_6$. Fix a Cartan subalgebra of $\mathfrak{g}$ and a simple root system $\{\alpha_i\mid1\leq i\leq6\}$, the Dynkin diagram of which is given in Figure 1.
\begin{figure}
\centering \scalebox{0.7}{\includegraphics{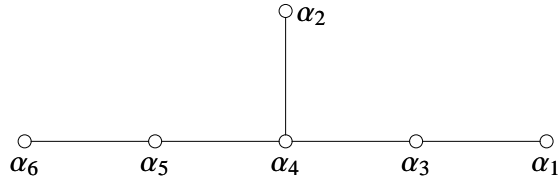}}
\caption{Dynkin diagram of $\mathrm{E}_6$.}
\end{figure}
For each root $\alpha$, denote by $H_\alpha$ its coroot, and denote by $X_\alpha$ the normalized root vector so that $[X_\alpha,X_{-\alpha}]=H_\alpha$. Moreover, one can normalize $X_\alpha$ appropriately such that\[\mathrm{Span}_\mathbb{R}\{X_\alpha-X_{-\alpha},\sqrt{-1}(X_\alpha+X_{-\alpha}),\sqrt{-1}H_\alpha\mid\alpha:\textrm{positive root}\}\cong\mathfrak{u}_0:=\mathfrak{e}_{6(-78)}\]is a compact real form of $\mathfrak{g}$ by \cite{Kn}. It is well known that\[\mathrm{Aut}\mathfrak{u}_0/\mathrm{Int}\mathfrak{u}_0\cong\mathrm{Aut}\mathfrak{g}/\mathrm{Int}\mathfrak{g}\]which is just the automorphism group of the Dynkin diagram.

The author follows the constructions of involutive automorphisms of $\mathfrak{u}_0$ in \cite{HY}. Let $\tau$ be the specific involutive automorphism of the Dynkin diagram defined by
\begin{eqnarray*}
\begin{array}{rclcrcl}
\tau(H_{\alpha_1})=H_{\alpha_6},&&\tau(X_{\pm\alpha_1})=X_{\pm\alpha_6},\\
\tau(H_{\alpha_2})=H_{\alpha_2},&&\tau(X_{\pm\alpha_2})=X_{\pm\alpha_2},\\
\tau(H_{\alpha_3})=H_{\alpha_5},&&\tau(X_{\pm\alpha_3})=X_{\pm\alpha_5},\\
\tau(H_{\alpha_4})=H_{\alpha_4},&&\tau(X_{\pm\alpha_4})=X_{\pm\alpha_4}.
\end{array}
\end{eqnarray*}
Let $\sigma_1=\mathrm{exp}(\sqrt{-1}\pi H_{\alpha_2})$, $\sigma_2=\mathrm{exp}(\sqrt{-1}\pi(H_{\alpha_1}+H_{\alpha_6}))$, $\sigma_3=\tau$, $\sigma_4=\tau\mathrm{exp}(\sqrt{-1}\pi H_{\alpha_2})$, where $\mathrm{exp}$ represents the exponential map from $\mathfrak{u}_0$ to $\mathrm{Aut}\mathfrak{u}_0$. Then $\sigma_1$, $\sigma_2$, $\sigma_3$, and $\sigma_4$ represent all conjugacy classes of involutions in $\mathrm{Aut}\mathfrak{u}_0$, which correspond to real forms $\mathfrak{e}_{6(2)}$, $\mathfrak{e}_{6(-14)}$, $\mathfrak{e}_{6(-26)}$, and $\mathfrak{e}_{6(6)}$.

Let $G=\mathrm{Aut}\mathfrak{u}_0$, and $G_0=\mathrm{Int}\mathfrak{u}_0$ be the identity component of $G$. From \cite{HY}, it is known that $(G_0)^{\sigma_3}\cong\mathrm{F}_{4(-52)}$, the compact Lie group of type $\mathrm{F}_4$, and there exist involutive automorphisms $\tau_1$ and $\tau_2$ of $\mathrm{F}_{4(-52)}$ such that $\mathfrak{f}_{4(-52)}^{\tau_1}\cong\mathfrak{sp}(3)\oplus\mathfrak{sp}(1)$ and $\mathfrak{f}_{4(-52)}^{\tau_2}\cong\mathfrak{so}(9)$, where $\mathfrak{f}_{4(-52)}$ denotes the compact Lie algebra of type $\mathrm{F}_4$.

By \cite{HY}, $\tau_1$, $\tau_2$, $\sigma_3\tau_1$, and $\sigma_3\tau_2$ represent all conjugacy classes of involutive automorphisms in $G^{\sigma_3}$ except for $\sigma_3$ and there are conjugacy relations in $G$: $\tau_1\sim_G\sigma_1$, $\tau_1\sim_G\sigma_2$, $\sigma_3\tau_1\sim_G\sigma_4$, and $\sigma_3\tau_2\sim_G\sigma_3$. Moreover, $((G_0)^{\sigma_3})^{\tau_1}\cong\mathrm{Sp}(3)\times\mathrm{Sp}(1)/\langle(-I_3,-1)\rangle$. Let $\mathbf{i}$, $\mathbf{j}$, and $\mathbf{k}$ denote the fundamental quaternion units, and then set $x_0=\sigma_3$, $x_1=\tau_1=(I_3,-1)$, $x_2=(\mathbf{i}I_3,\mathbf{i})$, $x_3=(\mathbf{j}I_3,\mathbf{j})$, $x_4=(\left(\begin{array}{ccc}-1&0&0\\0&-1&0\\0&0&1\end{array}\right),1)$, and $x_5=(\left(\begin{array}{ccc}-1&0&0\\0&1&0\\0&0&-1\end{array}\right),1)$.

For a integer pair $(r,s)$ with $r\leq2$ and $s\leq3$, define\[F_{r,s}:=\langle x_0,x_1,\cdots,x_s,x_4,x_5,\cdots,x_{r+3}\rangle\]and\[F'_{r,s}:=\langle x_1,x_2,\cdots,x_s,x_4,x_5,\cdots,x_{r+3}\rangle.\]

On the other hand, $(G_0)^{\sigma_4}\cong\mathrm{Sp}(4)/\langle-I_4\rangle$. Let $\eta_1=\mathbf{i}I$, $\eta_2=\left(\begin{array}{cc}-I_2&0\\0&I_2\end{array}\right)$, and $\eta_3=\left(\begin{array}{cc}-1&0\\0&I_3\end{array}\right)$. By \cite{HY}, $\eta_1$, $\eta_2$, $\eta_3$, $\sigma_4\eta_1$, $\sigma_4\eta_2$, and $\sigma_4\eta_3$ represent all conjugacy classes of involutive automorphisms in $G^{\sigma_4}$ except for $\sigma_3$ and there are conjugacy relations in $G$: $\eta_1\sim_G\eta_3\sim_G\sigma_1$, $\eta_2\sim_G\sigma_2$, $\sigma_4\eta_1\sim_G\sigma_4\eta_2\sim_G\sigma_4$, and $\sigma_4\eta_3\sim_G\sigma_3$. Set $y_0=\sigma_4$, $y_1=\mathbf{i}I_4$, $y_2=\mathbf{j}I_4$, $y_3=\left(\begin{array}{cc}-I_2&0\\0&I_2\end{array}\right)$, $y_4=\left(\begin{array}{cc}0&I_2\\I_2&0\end{array}\right)$, $y_5=\left(\begin{array}{cccc}1&0&0&0\\0&-1&0&0\\0&0&1&0\\0&0&0&-1\end{array}\right)$, and $y_6=\left(\begin{array}{cccc}0&1&0&0\\1&0&0&0\\0&0&0&1\\0&0&1&0\end{array}\right)$.

For a integer pair $(u,v,r,s)$ with $u+v\leq1$ and $r+s\leq2$, define\[F_{u,v,r,s}:=\langle y_0,y_1,\cdots,y_{u+2v},y_3,y_4,\cdots,y_{2+2s},y_{3+2s},y_{5+2s},\cdots,y_{1+2r+2s}\rangle\]and\[F'_{u,v,r,s}:=\langle y_1,y_2,\cdots,y_{u+2v},y_3,y_4,\cdots,y_{2+2s},y_{3+2s},y_{5+2s},\cdots,y_{1+2r+2s}\rangle.\]According to \cite[Proposition 6.3 \& Proposition 6.5]{Y}, each elementary abelian 2-subgroup in $\mathrm{Aut}\mathfrak{u}_0$ is conjugate to one of the groups in the one family of $F_{r,s}$, $F'_{r,s}$, $F_{u,v,r,s}$, and $F'_{u,v,r,s}$, and all groups in the families of $F_{r,s}$, $F'_{r,s}$, $F_{u,v,r,s}$, and $F'_{u,v,r,s}$ are pairwisely non-conjugate.
\subsection{Elements conjugate to $\sigma_2$}
In this part, the author needs to determine which involutive automorphisms are conjugate to $\sigma_2$ in $G$ because the noncompact dual $\mathfrak{g}_0=\mathfrak{e}_{6(-14)}$ of $\mathfrak{u}_0$ corresponding to $\sigma_2$ is of Hermitian type.

Firstly, consider the elements in $F_{r,s}$ and $F'_{r,s}$. In fact, one only needs to check the elements in the largest group $F_{2,3}=\langle x_0,x_1,x_2,x_3,x_4,x_5\rangle$.
\begin{lemma}\label{18}
The elements in $F_{2,3}$ conjugate to $\sigma_2$ in $\mathrm{Aut}\mathfrak{u}_0$ are exactly $x_4$, $x_5$, and $x_4x_5$.
\end{lemma}
\begin{proof}
Let $H=\mathrm{Sp}(3)\times\mathrm{Sp}(1)/\langle(-I_3,-1)\rangle$ for convenience. Consider the following four points.
\begin{enumerate}[$\bullet$]
\item Obviously, $x_0\sim_G\sigma_3\nsim_G\sigma_2$. Moreover, $x_0$ is an outer automorphism while $x_i$ for $1\leq i\leq5$ are all inner automorphisms. Thus if $\omega\in F'_{2,3}=\langle x_1,x_2,x_3,x_4,x_5\rangle$, then $x_3\omega$ is an outer automorphism which cannot be conjugate to the inner automorphism $\sigma_2$. Hence the author only needs to check the elements in the biggest group $F'_{2,3}=\langle x_1,x_2,x_3,x_4,x_5\rangle$.
\item Let $S_1=\{x_1,x_2,x_3,x_1x_2,x_1x_3,x_2x_3,x_1x_2x_3\}$. Firstly, $x_1\sim_G\sigma_1$ by \cite{Y}. Secondly, $x_2$ is conjugate to $\tau_1$ in $\mathrm{F}_{4(-52)}$ by \cite{HY}, and it follows that $x_2\sim_G\sigma_1$. Thirdly, since it is well known that $\pm\mathbf{i}$, $\pm\mathbf{j}$, and $\pm\mathbf{k}$ are pairwisely conjugate in the group ring of the quoternion group, it is immediate that $x_2\sim_Hx_3\sim_Hx_2x_3\sim_Hx_1x_2\sim_Hx_1x_3\sim_Hx_1x_2x_3$. In particular, it is conclude that all elements in $S_1$ are conjugate to $\sigma_1$ in $\mathrm{Aut}\mathfrak{u}_0$.
\item Let $S_2=\{x_4,x_5,x_4x_5\}$. Firstly, $x_4$ is conjugate to $\tau_2$ in $\mathrm{F}_{4(-52)}$ by \cite{HY}, and it follows that $x_4\sim_G\sigma_2$. Secondly, it is obvious that $x_4\sim_Hx_5\sim_Hx_4x_5$. In particular, it is conclude that all elements in $S_2$ are conjugate to $\sigma_2$ in $\mathrm{Aut}\mathfrak{u}_0$.
\item Consider $\omega_1\omega_2$ where $\omega_1\in S_1$ and $\omega_2\in S_2$. Firstly, $x_1x_4=(\left(\begin{array}{ccc}1&0&0\\0&1&0\\0&0&-1\end{array}\right),1)$ in $\mathrm{Sp}(3)\times\mathrm{Sp}(1)/\langle(-I_3,-1)\rangle$, which is conjugate to $\tau_1$ in $\mathrm{F}_{4(-52)}$ by \cite{HY}, and it follows that $x_1x_4\sim_G\sigma_1$. Similary, one has $x_1x_5\sim_Gx_1x_4x_5\sim_G\sigma_1$. Secondly, because $\mathbf{i}$ is conjugate to $-\mathbf{i}$ in the group ring of the quoternion group, it is easy to see that $x_2x_4\sim_Hx_1x_2x_4\sim_Hx_2$, and hence $x_2x_4\sim_Gx_1x_2x_4\sim_G\sigma_1$. In similar way, one can show that $\omega_1\omega_2\sim_G\sigma_1$ for all $\omega_1\in S_1$ and $\omega_2\in S_2$.
\end{enumerate}
Based on the four points, the conclusion holds.
\end{proof}
Secondly, consider the elements in $F_{u,v,r,s}$ and $F'_{u,v,r,s}$. In fact, one only needs to check the elements in the largest group $F_{0,1,0,2}=\langle y_0,y_1,y_2,y_3,y_4,y_5,y_6\rangle$. Similarly, exclude the outer automorphisms,  and one only needs to check the elements in $F'_{0,1,0,2}=\langle y_1,y_2,y_3,y_4,y_5,y_6\rangle$.

In order to find the elements in $F'_{0,1,0,2}$ which are conjugate to $\sigma_2$, the author uses a function constructed in \cite[Definition 6.1]{Y}. Define\[\mu:F'_{0,1,0,2}\longrightarrow\{\pm1\}\]given by\[\mu(y)=\begin{cases}-1,\textrm{if }y\sim_G\sigma_1,\\1,\textrm{if }y\sim_G\sigma_2\textrm{ or }y=1.\end{cases}\]

By \cite[Lemma 2.13 \& Lemma 6.6]{Y}, there is another equivalent definition for the function $\mu$ on $F'_{0,1,0,2}$. For $y\in F'_{0,1,0,2}\subseteq(G_0)^{\sigma_4}\cong\mathrm{Sp}(4)/\langle-I_4\rangle$, choose $A\in\mathrm{Sp}(4)$ representing $y$, and then $A^2=\epsilon_AI_4$ for some $\epsilon_A=\pm1$. Obviously, the value $\epsilon_A$ does not depend on the choice of $A$. Then define $\mu(x):=\epsilon_A$.

\begin{lemma}\label{19}
The values of the function $\mu$ on $F'_{0,1,0,2}$ are given as follows:
\begin{enumerate}[(i)]
\item $\mu(y_1)=\mu(y_2)=\mu(y_1y_2)=\mu(y_3y_4)=\mu(y_5y_6)=-1$ and $\mu(y_3)=\mu(y_4)=\mu(y_5)=\mu(y_6)=1$.
\item For $x,y\in F'_{0,1,0,2}$, if $xy=yx$, then $\mu(xy)=\mu(x)\mu(y)$. In particular, if $x\in\langle y_1,y_2\rangle$, $y\in\langle y_3,y_4\rangle$, and $z\in\langle y_5,y_6\rangle$, then $\mu(xyz)=\mu(x)\mu(y)\mu(z)$.
\end{enumerate}
\end{lemma}
\begin{proof}
The conclusion follows from the equivalent definition of $\mu$ immediately.
\end{proof}
\begin{lemma}\label{20}
There are totally 27 elements in $F'_{0,1,0,2}$, which are conjugate to $\sigma_2$ in $G$.
\end{lemma}
\begin{proof}
According to the definition of $\mu$, elements conjugate to $\sigma_2$ in $F'_{0,1,0,2}$ are exactly $\mu^{-1}(1)\setminus\{1\}$. Thus, by \lemmaref{19}, there are 27 such elements: $y_3$, $y_4$, $y_5$, $y_6$, $y_3y_5$, $y_3y_6$, $y_4y_5$, $y_4y_6$, $y_1y_3y_4$, $y_1y_5y_6$, $y_2y_3y_4$, $y_2y_5y_6$, $y_1y_2y_3y_4$, $y_1y_2y_5y_6$, $y_1y_3y_4y_5$, $y_1y_3y_4y_6$, $y_1y_3y_5y_6$, $y_1y_4y_5y_6$, $y_2y_3y_4y_5$, $y_2y_3y_4y_6$, $y_2y_3y_5y_6$, $y_2y_4y_5y_6$, $y_3y_4y_5y_6$, $y_1y_2y_3y_4y_5$, $y_1y_2y_3y_4y_6$, $y_1y_2y_3y_5y_6$, and $y_1y_2y_4y_5y_6$.
\end{proof}
\begin{proposition}\label{21}
There are 7 conjugacy classes of the elementary abelian 2-subgroups of rank 3 in $G$, which contain elements conjugate to $\sigma_2$, represented by $\langle x_0,x_1,x_4\rangle$, $\langle x_0,x_4,x_5\rangle$, $\langle x_1,x_2,x_4\rangle$, $\langle x_1,x_4,x_5\rangle$, $\langle y_0,y_3,y_4\rangle$, $\langle y_1,y_3,y_4\rangle$, and $\langle y_3,y_4,y_5\rangle$ respectively.
\end{proposition}
\begin{proof}
This follows from \lemmaref{18}, \lemmaref{20}, and \cite[Proposition 6.3 \& Proposition 6.5]{Y} immediately.
\end{proof}
\subsection{Klein four symmetric subalgebras of $\mathfrak{e}_{6(-14)}$}
Let $\mathfrak{g}_0$ be a non-compact simple Lie algebra. For each involutive automorphism $\sigma$ of $\mathfrak{g}_0$, there exists a Cartan involution $\theta$ of $\mathfrak{g}_0$, which commutes with $\sigma$. Denote by $\mathfrak{g}:=\mathfrak{g}_0+\sqrt{-1}\mathfrak{g}_0$ the complexification of $\mathfrak{g}_0$, and write $\mathfrak{u}_0:=\mathfrak{g}_0^\theta+\sqrt{-1}\mathfrak{g}_0^{-\theta}$ for the compact dual of $\mathfrak{g}_0$. Extend $\sigma$ to the unique holomorphic involutive automorphism of $\mathfrak{g}$, restrict it to $\mathfrak{u}_0$, and $\sigma$ becomes an involutive automorphism of $\mathfrak{u}_0$. Thus, any involutive automorphism $\sigma$ of $\mathfrak{g}_0$ gives a pair $(\mathfrak{u}_0,\sigma)$, where $\mathfrak{u}_0$ is a compact real form of $\mathfrak{g}$ and $\sigma$ is regarded as an involutive automorphism of $\mathfrak{u}_0$. If $\sigma'$ is conjugate to $\sigma$ in $\mathrm{Aut}\mathfrak{g}_0$, which gives another pair $(\mathfrak{u}'_0,\sigma')$, then there exists an element $g\in\mathrm{Aut}\mathfrak{g}$ such that $g\cdot\mathfrak{u}_0=\mathfrak{u}'_0$ and $g^{-1}\sigma g=\sigma'$; namely, the two pairs $(\mathfrak{u}_0,\sigma)$ and $(\mathfrak{u}'_0,\sigma')$ are conjugate by elements in $\mathrm{Aut}\mathfrak{g}$. Thus, up to conjugations of compact real forms of $\mathfrak{g}$, there exists a map\[\pi:\{\textrm{Conjugacy classes of }\mathrm{Aut}\mathfrak{g}_0\}\longrightarrow\{\textrm{Conjugacy classes of }\mathrm{Aut}\mathfrak{u}_0\}\]which is surjective but not necessarily injective.

If $\Gamma$ is a finite abelian subgroup of $\mathrm{Aut}\mathfrak{g}_0$, there exists a Cartan involution $\theta\in\mathrm{Aut}\mathfrak{g}_0$ centralizing $\Gamma$ by \propositionref{2}. In the same way, $\Gamma$ becomes a finite abelian subgroup of $\mathrm{Aut}\mathfrak{u}_0$. Thus, any finite abelian subgroup of $\mathrm{Aut}\mathfrak{g}_0$ is obtained from a finite abelian subgroup of $\mathfrak{u}_0$.

On the other hand, fix an involutive automorphism $\theta$ of $\mathfrak{u}_0$, and by holomorphic extension and restriction, $\theta$ is a Cartan involution of a noncompact dual $\mathfrak{g}_0$ of $\mathfrak{u}_0$. Let $\Theta:\mathrm{Aut}\mathfrak{u}_0\rightarrow\mathrm{Aut}\mathfrak{u}_0$ be given by $f\mapsto\theta^{-1}f\theta$, whose differential is $\theta$ on $\mathfrak{u}_0$. Suppose that $\Gamma_1$ and $\Gamma_2$ are two finite abelian subgroups of $\mathrm{Aut}\mathfrak{u}_0$, which are centralized by $\theta$. If $\Gamma_2=g^{-1}\Gamma_1g$ for some $g\in(\mathrm{Aut}\mathfrak{u}_0)^\Theta$, then $\Gamma_1$ and $\Gamma_2$ is conjugate in $\mathrm{Aut}\mathfrak{g}_0$ because $(\mathrm{Aut}\mathfrak{u}_0)^\Theta$ is contained in $\mathrm{Aut}\mathfrak{g}_0$.

Thus, for a compact Lie algebra $\mathfrak{u}_0$, a pair $(\theta,\Gamma)$ with $\theta$ an involutive automorphism of $\mathfrak{u}_0$ and $\Gamma$ a Klein four subgroup of $\mathrm{Aut}\mathfrak{u}_0$ such that $\theta\notin\Gamma$ and $\theta$ centralizes $\Gamma$, gives a Klein four symmetric subalgebra of the noncompact Lie algebra $\mathfrak{g}_0=\mathfrak{u}_0^\theta+\sqrt{-1}\mathfrak{u}_0^{-\theta}$. Moreover, if two such pairs are conjugate $\mathrm{Aut}\mathfrak{u}_0$, then they give a same Klein four symmetric pair up to isomorphism. Here, the pairs $(\theta_1,\Gamma_1)$ and $(\theta_2,\Gamma_2)$ are said to be conjugate in $\mathrm{Aut}\mathfrak{u}_0$ if there exists an element $g\in\mathrm{Aut}\mathfrak{u}_0$ such that $\theta_2=g^{-1}\theta_1g$ and $\Gamma_2=g^{-1}\Gamma_1g$.

In order to find all the Klein four symmetric pair of holomorphic type for $\mathfrak{g}_0=\mathfrak{e}_{6(-14)}$, according to the argument above, the author needs to find all the conjugacy classes of the pairs $(\theta,\Gamma)$ with $\theta\in G$ involutive automorphisms and $\Gamma\subseteq G$ Klein four subgroups such that
\begin{enumerate}[(i)]
\item $\theta\notin\Gamma$;
\item $\theta\sim_G\sigma_2$;
\item $\theta$ centralizes $\Gamma$;
\item every element $\sigma\in\Gamma$ is identity on the center of $\mathfrak{u}_0^\theta$; namely, each $\sigma\in\Gamma$ gives a symmetric pair of holomorphic type for $\mathfrak{g}_0=\mathfrak{u}_0^\theta+\sqrt{-1}\mathfrak{u}_0^{-\theta}$.
\end{enumerate}
Because of the items (i), (ii), and (iii), one knows immediately that each subgroup generated by such $\theta$ and $\Gamma$ must be contained in the some elementary abelian 2-subgroup of rank 3 in $G$, i.e., one of the 7 subgroups listed in \propositionref{21} up to conjugation.

On the other hand, notice that $\pi$ is a surjection from the set of conjugacy classes of $\mathfrak{g}_0$ to the set of conjugacy classes of $\mathfrak{u}_0$. According to the classification of symmetric pairs, it is known that the cardinality of pre-image $|\pi^{-1}([\sigma_1])|$ of $[\sigma_1]$ is 2, while that of $[\sigma_2]$ is 3. Write $\sigma_{1,1}$ and $\sigma_{1,2}$ for the two representatives of the two conjugacy classes in $\mathrm{Aut}\mathfrak{g}_0$ corresponding to $[\sigma_1]$, and write $\sigma_{2,1}$, $\sigma_{2,2}$, and $\sigma_{2,3}$ for the two representatives of the two conjugacy classes in $\mathrm{Aut}\mathfrak{g}_0$ corresponding to $[\sigma_2]$. Here, $\mathfrak{g}_0^{\sigma_{1,1}}=\mathfrak{su}(4,2)\oplus\mathfrak{su}(2)$, $\mathfrak{g}_0^{\sigma_{1,2}}=\mathfrak{su}(5,1)\oplus\mathfrak{sl}(2,\mathbb{R})$, $\mathfrak{g}_0^{\sigma_{2,1}}=\mathfrak{s0}(8,2)\oplus\mathfrak{so}(2)$, $\mathfrak{g}_0^{\sigma_{2,2}}=\mathfrak{so}^*(10)\oplus\mathfrak{so}(2)$, and $\mathfrak{g}_0^{\sigma_{2,3}}=\mathfrak{so}(10)\oplus\mathfrak{so}(2)$. According to the classification of symmetric pairs of holomorphic type \cite[Table C.2]{KO}, $\sigma_{1,1}$, $\sigma_{1,2}$, $\sigma_{2,1}$, $\sigma_{2,2}$, and $\sigma_{2,3}$ give all the symmetric pairs of holomorphic type for $\mathfrak{g}_0$. Hence, taking the item (iii) into consideration, there is no need to consider the elementary abelian 2-subgroups of rank 3 in $G$, which contain outer automorphisms. Thus, a pair $(\theta,\Gamma)$ satisfies the above four requirements if and only if $\theta\sim_G\sigma_2$ and the group $\langle\theta,\Gamma\rangle$ generated by $\theta$ and the elements in $\Gamma$ is conjugate to one of following abelian elementary 2-subgroups of rank 3:
\begin{enumerate}[(i)]
\item $\langle x_1,x_2,x_4\rangle$,
\item $\langle x_1,x_4,x_5\rangle$,
\item $\langle y_1,y_3,y_4\rangle$,
\item $\langle y_3,y_4,y_5\rangle$.
\end{enumerate}
In order to compute the explicit Klein four symmetric pairs of holomorphic type for $\mathfrak{g}_0$, for each elementary abelian 2-subgroup of rank 3 listed above, the author needs to take three steps.
\begin{enumerate}[Step 1:]
\item Determine the conjugacy classes $(\theta,\Gamma)$ in $\mathrm{Aut}\mathrm{u}_0$;
\item Compute $\mathfrak{u}_0^\Gamma$ which is the compact dual of $\mathfrak{g}_0^\Gamma$;
\item Compute $\mathfrak{u}_0^{\langle\theta,\Gamma\rangle}$ which is a maximal compact subalgebra of $\mathfrak{g}_0$.
\end{enumerate}
After the three steps, the Klein four symmetric subalgebras $\mathfrak{g}_0^\Gamma$ are totally determined.
\begin{lemma}\label{22}
Every $(\theta,\Gamma)$ in $\mathrm{Aut}\mathrm{u}_0$, such that $\theta\sim_G\sigma_2$ and the group $\langle\theta,\Gamma\rangle$ generated by $\theta$ and $\Gamma$ is conjugate to one of the above 4 elementary abelian 2-subgroups of rank 3, is conjugate to one of the following pairs.
\begin{enumerate}[(i)]
\item $\langle x_1,x_2,x_4\rangle$:
\begin{enumerate}[$\bullet$]
\item $\theta=x_4$, $\Gamma=\langle x_1,x_2\rangle$.
\end{enumerate}
\item $\langle x_1,x_4,x_5\rangle$:
\begin{enumerate}[$\bullet$]
\item $\theta=x_5$, $\Gamma=\langle x_1,x_4\rangle$.
\end{enumerate}
\item $\langle y_1,y_3,y_4\rangle$:
\begin{enumerate}[$\bullet$]
\item $\theta=y_3$, $\Gamma=\langle y_1,y_4\rangle$;
\item $\theta=y_3$, $\Gamma=\langle y_1y_3,y_4\rangle$;
\item $\theta=y_3$, $\Gamma=\langle y_1y_3,y_3y_4\rangle$;
\end{enumerate}
\item $\langle y_3,y_4,y_5\rangle$:
\begin{enumerate}[$\bullet$]
\item $\theta=y_5$, $\Gamma=\langle y_3,y_4\rangle$;
\item $\theta=y_3$, $\Gamma=\langle y_4,y_5\rangle$;
\item $\theta=y_3$, $\Gamma=\langle y_3y_4,y_5\rangle$;
\item $\theta=y_3$, $\Gamma=\langle y_4,y_3y_5\rangle$.
\end{enumerate}
\end{enumerate}
\end{lemma}
\begin{proof}
If the group $\langle\theta,\Gamma\rangle$ is conjugate to $\langle x_1,x_2,x_4\rangle$, then the pair $(\theta,\Gamma)$ is conjugate to one of $(x_4,\langle x_1,x_2\rangle)$, $(x_4,\langle x_1x_4,x_2\rangle)$, $(x_4,\langle x_1,x_2x_4\rangle)$, and $(x_4,\langle x_1x_4,x_2x_4\rangle)$. The author needs to show that they are all conjugate. In fact, there is an automorphism $f$ of the group $\langle x_1,x_2,x_4\rangle$, given by $f(x_1)=x_1x_4$, $f(x_2)=x_2$, and $f(x_4)=x_4$. By \lemmaref{18}, it is immediate that $f(x)\sim_Gx$ for all $x\in\langle x_1,x_2,x_4\rangle$. Then by \cite[Proposition 6.9]{Y}, there exists an element $g\in G$ such that $f(x)=g^{-1}xg$ for all $x\in\langle x_1,x_2,x_4\rangle$, and this shows that $(x_4,\langle x_1,x_2\rangle)$ and $(x_4,\langle x_1x_4,x_2\rangle)$ are conjugate. Similarly, both $(x_4,\langle x_1x_4,x_2\rangle)$ and $(x_4,\langle x_1,x_2x_4\rangle)$ are conjugate to $(x_4,\langle x_1,x_2\rangle)$.

If the group $\langle\theta,\Gamma\rangle$ is conjugate to $\langle x_1,x_4,x_5\rangle$, apply \lemmaref{18} and \cite[Proposition 6.9]{Y} in the same way, one shows that $(\theta,\Gamma)$ is conjugate to $(x_5,\langle x_1,x_4\rangle)$.

If $\langle\theta,\Gamma\rangle$ is conjugate to $\langle y_1,y_3,y_4\rangle$, then the pair $(\theta,\Gamma)$ is conjugate to one of $(y_3,\langle y_1,y_4\rangle)$, $(y_3,\langle y_1y_3,y_4\rangle)$, $(y_3,\langle y_1,y_3y_4\rangle)$, $(y_3,\langle y_1y_3,y_3y_4\rangle)$, $(y_4,\langle y_1,y_3\rangle)$, $(y_4,\langle y_1y_4,y_3\rangle)$, $(y_4,\langle y_1,y_3y_4\rangle)$, $(y_4,\langle y_1y_4,y_3y_4\rangle)$, $(y_1y_3y_4,\langle y_1,y_3\rangle)$, $(y_1y_3y_4,\langle y_1,y_4\rangle)$, $(y_1y_3y_4,\langle y_3,y_4\rangle)$, and $(y_1y_3y_4,\langle y_1y_4,y_3y_4\rangle)$. One may use \lemmaref{20} and \cite[Proposition 6.9]{Y} to show that the pairs in each of the following items are conjugate:
\begin{enumerate}[$\bullet$]
\item $(y_3,\langle y_1,y_4\rangle)$, $(y_3,\langle y_1,y_3y_4\rangle)$, $(y_4,\langle y_1,y_3\rangle)$, $(y_4,\langle y_1,y_3y_4\rangle)$, $(y_1y_3y_4,\langle y_1,y_3\rangle)$, and $(y_1y_3y_4,\langle y_1,y_4\rangle)$;
\item $(y_3,\langle y_1y_3,y_4\rangle)$, $(y_4,\langle y_1y_4,y_3\rangle)$, and $(y_1y_3y_4,\langle y_3,y_4\rangle)$;
\item $(y_3,\langle y_1y_3,y_3y_4\rangle)$, $(y_4,\langle y_1y_4,y_3y_4\rangle)$, and $(y_1y_3y_4,\langle y_1y_4,y_3y_4\rangle)$.
\end{enumerate}
Moreover, since the groups $\langle y_1,y_4\rangle$, $\langle y_1y_3,y_4\rangle$, and $\langle y_1y_3,y_3y_4\rangle$ are pairwisely non-conjugate by \cite[Table 4]{HY}, the three items above correspond to three conjugacy classes exactly.

If If $\langle\theta,\Gamma\rangle$ is conjugate to $\langle y_3,y_4,y_5\rangle$, in the same way, one shows that $(\theta,\Gamma)$ is conjugate to at least one of $(y_5,\langle y_3,y_4\rangle)$, $(y_3,\langle y_4,y_5\rangle)$, $(y_3,\langle y_3y_4,y_5\rangle)$, and $(y_3,\langle y_4,y_3y_5\rangle)$. Moreover, $(y_5,\langle y_3,y_4\rangle)$, $(y_3,\langle y_4,y_5\rangle)$, $(y_3,\langle y_3y_4,y_5\rangle)$, and $(y_3,\langle y_4,y_3y_5\rangle)$ are pairwisely non-conjugate by \cite[Table 4]{HY} except that $(y_5,\langle y_3,y_4\rangle)$ and $(y_3,\langle y_4,y_3y_5\rangle)$ may be conjugate.
\end{proof}
By \lemmaref{22}, each Klein four symmetric subalgebra of $\mathfrak{g}_0$ is isomorphic to one of $\mathfrak{g}_0^{\langle x_1,x_2\rangle}$, $\mathfrak{g}_0^{\langle x_1,x_4\rangle}$, $\mathfrak{g}_0^{\langle y_1,y_4\rangle}$, $\mathfrak{g}_0^{\langle y_1y_3,y_4\rangle}$, $\mathfrak{g}_0^{\langle y_1y_3,y_3y_4\rangle}$, $\mathfrak{g}_0^{\langle y_3,y_4\rangle}$, $\mathfrak{g}_0^{\langle y_4,y_5\rangle}$, $\mathfrak{g}_0^{\langle y_3y_4,y_5\rangle}$, and $\mathfrak{g}_0^{\langle y_4,y_3y_5\rangle}$, whose compact duals are give according to \cite[Table 4]{HY} as follows.
\begin{enumerate}[$\bullet$]
\item $\mathfrak{u}_0^{\langle x_1,x_2\rangle}\cong\mathfrak{u}_0^{\langle y_1y_3,y_3y_4\rangle}\cong2\mathfrak{su}(3)\oplus2(\sqrt{-1}\mathbb{R})$
\item $\mathfrak{u}_0^{\langle x_1,x_4\rangle}\cong\mathfrak{u}_0^{\langle y_1,y_4\rangle}\cong\mathfrak{u}_0^{\langle y_3y_4,y_5\rangle}\cong\mathfrak{su}(4)\oplus2\mathfrak{su}(2)\oplus\sqrt{-1}\mathbb{R}$
\item $\mathfrak{u}_0^{\langle y_1y_3,y_4\rangle}\cong\mathfrak{u}_0^{\langle y_3,y_4\rangle}\cong\mathfrak{u}_0^{\langle y_4,y_3y_5\rangle}\cong\mathfrak{su}(5)\oplus2(\sqrt{-1}\mathbb{R})$
\item $\mathfrak{u}_0^{\langle y_4,y_5\rangle}\cong\mathfrak{so}(8)\oplus2(\sqrt{-1}\mathbb{R})$
\end{enumerate}
Finally, the author needs to determine all $\mathfrak{u}_0^{\langle\theta,\Gamma\rangle}$ listed in \lemmaref{22}. It is very hard to compute them directly because all $x_i$ and $y_j$ are inner automorphisms taken from $G^{\sigma_3}$ and $G^{\sigma_4}$, but both $\sigma_3$ and $\sigma_4$ are outer automorphisms. However, by \cite[Proposition 6.9]{Y} again, one may find the groups isomorphic to those $\langle\theta,\Gamma\rangle$ in $G^{\sigma_1}$ or $G^{\sigma_2}$.

According to \cite{HY}, it is known that $G^{\sigma_1}\cong(\mathrm{SU}(6)\times\mathrm{Sp}(1)/\langle(e^{\frac{2\pi\sqrt{-1}}{3}}I_6,1),(-I_6,-1)\rangle)\rtimes\langle\tau\rangle$, where $\sigma_1=(I_6,-1)$ and $\tau(X,Y)=(\left(\begin{array}{cc}0&-I_3\\I_3&0\end{array}\right)\overline{X}\left(\begin{array}{cc}0&I_3\\-I_3&0\end{array}\right),Y)$ for $(X,Y)\in\mathfrak{su}(6)\oplus\mathfrak{sp}(1)$, with $\overline{X}$ being the natural complex conjugation of $X$ in $\mathfrak{su}(6)$. Then in $G$, one has $(\left(\begin{array}{cc}-I_2&0\\0&I_4\end{array}\right),1)\sim_G(\left(\begin{array}{cc}\sqrt{-1}I_3&0\\0&-\sqrt{-1}I_3\end{array}\right),\mathbf{i})\sim_G \sigma_1$ and $(\left(\begin{array}{cc}-I_4&0\\0&I_2\end{array}\right),1)\sim_G(\left(\begin{array}{cc}\sqrt{-1}I_5&0\\0&-\sqrt{-1}I_1\end{array}\right),\mathbf{i})\sim_G \sigma_2$. By tedious computations, one has group homomorphisms as follows.
\begin{enumerate}[(i)]
\item $\langle x_1,x_2,x_4\rangle\cong\langle(I_6,-1),(\left(\begin{array}{cc}\sqrt{-1}I_3&0\\0&-\sqrt{-1}I_3\end{array}\right),\mathbf{i}), (\left(\begin{array}{cccc}-I_2&0&0&0\\0&I_1&0&0\\0&0&-I_2&0\\0&0&0&I_1\end{array}\right),1)\rangle$ given by $x_1\mapsto(I_6,-1)$, $x_2\mapsto(\left(\begin{array}{cc}\sqrt{-1}I_3&0\\0&-\sqrt{-1}I_3\end{array}\right),\mathbf{i})$, and $x_4\mapsto(\left(\begin{array}{cccc}-I_2&0&0&0\\0&I_1&0&0\\0&0&-I_2&0\\0&0&0&I_1\end{array}\right),1)$;
\item $\langle x_1,x_4,x_5\rangle\cong\langle(I_6,-1),(\left(\begin{array}{cc}-I_2&0\\0&I_4\end{array}\right),1), (\left(\begin{array}{cc}-I_4&0\\0&I_2\\\end{array}\right),1)\rangle$ given by $x_1\mapsto(I_6,-1)$, $x_1x_4\mapsto(\left(\begin{array}{cc}-I_2&0\\0&I_4\end{array}\right),1)$, and $x_5\mapsto(\left(\begin{array}{cc}-I_4&0\\0&I_2\\\end{array}\right),1)$;
\item $\langle y_1,y_3,y_4\rangle\cong\langle(I_6,-1),(\left(\begin{array}{cc}\sqrt{-1}I_3&0\\0&-\sqrt{-1}I_3\end{array}\right),\mathbf{i}), (\left(\begin{array}{cc}\sqrt{-1}I_5&0\\0&-\sqrt{-1}I_1\end{array}\right),\mathbf{i})\rangle$ given by $y_1y_3\mapsto(I_6,-1)$, $y_1y_4\mapsto(\left(\begin{array}{cc}\sqrt{-1}I_3&0\\0&-\sqrt{-1}I_3\end{array}\right),\mathbf{i})$, and $y_4\mapsto(\left(\begin{array}{cc}\sqrt{-1}I_5&0\\0&-\sqrt{-1}I_1\end{array}\right),\mathbf{i})$;
\item $\langle y_3,y_4,y_5\rangle\cong\langle(I_6,-1),(\left(\begin{array}{cc}-I_4&0\\0&I_2\end{array}\right),1), (\left(\begin{array}{cc}\sqrt{-1}I_5&0\\0&-\sqrt{-1}I_1\end{array}\right),\mathbf{i})\rangle$ given by $y_3y_4\mapsto(I_6,-1)$, $y_5\mapsto(\left(\begin{array}{cc}-I_4&0\\0&I_2\end{array}\right),1)$, and $y_3\mapsto(\left(\begin{array}{cc}\sqrt{-1}I_5&0\\0&-\sqrt{-1}I_1\end{array}\right),\mathbf{i})$.
\end{enumerate}
\begin{lemma}\label{23}
The following isomorphisms of subalgebras in $\mathfrak{u}_0$ hold.
\begin{enumerate}[(i)]
\item $\mathfrak{u}_0^{\langle x_1,x_2,x_4\rangle}\cong2\mathfrak{su}(2)\oplus4(\sqrt{-1}\mathbb{R})$;
\item $\mathfrak{u}_0^{\langle x_1,x_4,x_5\rangle}\cong4\mathfrak{su}(2)\oplus2(\sqrt{-1}\mathbb{R})$;
\item $\mathfrak{u}_0^{\langle y_1,y_3,y_4\rangle}\cong\mathfrak{su}(3)\oplus\mathfrak{su}(2)\oplus3(\sqrt{-1}\mathbb{R})$;
\item $\mathfrak{u}_0^{\langle y_3,y_4,y_5\rangle}\cong\mathfrak{su}(4)\oplus3(\sqrt{-1}\mathbb{R})$.
\end{enumerate}
\end{lemma}
\begin{proof}
Based on the group isomorphisms of elementary abelian 2-subgroups of rank 3 in $\mathrm{Aut}\mathfrak{u}_0$, make use of \cite[Proposition 6.9]{Y}, and the conclusions hold immediately by computations.
\end{proof}
\begin{theorem}\label{24}
There are totally 8 Klein for symmetric pairs of holomorphic type for $\mathfrak{e}_{6(-14)}$ up to conjugation:
\begin{enumerate}[(i)]
\item $(\mathfrak{e}_{6(-14)},2\mathfrak{su}(2,1)\oplus2(\sqrt{-1}\mathbb{R}))$;
\item $(\mathfrak{e}_{6(-14)},\mathfrak{su}(2,2)\oplus2\mathfrak{su}(2)\oplus\sqrt{-1}\mathbb{R})$;
\item $(\mathfrak{e}_{6(-14)},\mathfrak{su}(3,1)\oplus\mathfrak{su}(1,1)\oplus\mathfrak{su}(2)\oplus\sqrt{-1}\mathbb{R})$;
\item $(\mathfrak{e}_{6(-14)},\mathfrak{su}(3,2)\oplus2(\sqrt{-1}\mathbb{R}))$;
\item $(\mathfrak{e}_{6(-14)},\mathfrak{su}(2,1)\oplus\mathfrak{su}(3)\oplus2(\sqrt{-1}\mathbb{R}))$;
\item $(\mathfrak{e}_{6(-14)},\mathfrak{su}(4,1)\oplus2(\sqrt{-1}\mathbb{R}))$;
\item $(\mathfrak{e}_{6(-14)},\mathfrak{so}(6,2)\oplus2(\sqrt{-1}\mathbb{R}))$;
\item $(\mathfrak{e}_{6(-14)},2\mathfrak{su}(1,1)\oplus\mathfrak{su}(4)\oplus\sqrt{-1}\mathbb{R})$.
\end{enumerate}
\end{theorem}
\begin{proof}
The conclusion follows from \lemmaref{22}, \lemmaref{23}, and the corresponding Klein four symmetric subalgebras of $\mathfrak{e}_{6(-78)}$ given in \cite[Table 4]{HY}, which are also listed above.
\end{proof}
\subsection{Application to Representation Theory}
As an application of \theoremref{24}, recall \corollaryref{17} that Klein four symmetric pairs of holomorphic type are related to the pairs $(G,G')$ of real reductive groups such that there exists a unitary highest weight representation of $G$, which is $G'$-admissible.

Let $\{e_i\}_{i=1}^n$ be a orthonormal basis for $\mathbb{R}^n$ with respect to the standard inner product, such that $\mathrm{Pin}(n)$ is defined as a subgroup contained in the Clifford algebra $\mathrm{Cl}(\mathbb{R}^n)$. Write $c=e_1e_2\cdots e_n\in\mathrm{Pin}(n)$, and it is known that $c\in\mathrm{Spin}(n)$ if and only if $n$ is even; in this case, $c$ is in the center of $\mathrm{Spin}(n)$ and any noncompact dual $\mathrm{Spin}(m,n-m)$ for $1\leq m\leq n-1$.
\begin{theorem}\label{25}
There are totally 8 Klein for symmetric pairs of holomorphic type for $\mathrm{E}_{6(-14)}$ up to conjugation:
\begin{enumerate}[(i)]
\item $(\mathrm{E}_{6(-14)},\mathrm{SU}(2,1)^2\times\mathrm{U}(1)^2/\langle(e^{\frac{2\pi\sqrt{-1}}{3}}I_3,I_3,e^{\frac{2\pi\sqrt{-1}}{3}},1), (I_3,e^{\frac{2\pi\sqrt{-1}}{3}}I_3,e^{-\frac{2\pi\sqrt{-1}}{3}},1)\rangle)$;
\item $(\mathrm{E}_{6(-14)},\mathrm{SU}(2,2)\times\mathrm{SU}(2)^2\times\mathrm{U}(1)/\langle(\sqrt{-1}I_4,-I_2,I_2,\sqrt{-1}),(I_4,-I_2,-I_2,-1)\rangle)$;
\item $(\mathrm{E}_{6(-14)},\mathrm{SU}(3,1)\times\mathrm{SU}(1,1)\times\mathrm{SU}(2)\times\mathrm{U}(1)/\langle(\sqrt{-1}I_4,-I_2,I_2,\sqrt{-1}), (I_4,-I_2,-I_2,-1)\rangle)$;
\item $(\mathrm{E}_{6(-14)},\mathrm{SU}(3,2)\times\mathrm{U}(1)^2)$;
\item $(\mathrm{E}_{6(-14)},\mathrm{SU}(2,1)\times\mathrm{SU}(3)\times\mathrm{U}(1)^2/ \langle(e^{\frac{2\pi\sqrt{-1}}{3}}I_3,I_3,e^{\frac{2\pi\sqrt{-1}}{3}},1),(I_3,e^{\frac{2\pi\sqrt{-1}}{3}}I_3,e^{-\frac{2\pi\sqrt{-1}}{3}},1)\rangle)$;
\item $(\mathrm{E}_{6(-14)},\mathrm{SU}(4,1)\times\mathrm{U}(1)^2)$;
\item $(\mathrm{E}_{6(-14)},\mathrm{Spin}(6,2)\times\mathrm{U}(1)^2/\langle(-1,-1,1),(c,1,-1)\rangle)$;
\item $(\mathrm{E}_{6(-14)},\mathrm{SU}(1,1)^2\times\mathrm{SU}(4)\times\mathrm{U}(1)/\langle(-I_2,I_2,\sqrt{-1}I_4,\sqrt{-1}),(-I_2,-I_2,I_4,-1)\rangle)$.
\end{enumerate}
\end{theorem}
\begin{proof}
Since $\mathrm{E}_{6(-14)}=\mathrm{Int}\mathfrak{e}_{6(-14)}$ is connected, the conclusion follows from \theoremref{24} and \cite[Table 6]{HY}.
\end{proof}
\begin{corollary}\label{26}
If $(G,G')$ is one of the pairs listed in \theoremref{25}, then any irreducible unitary highest weight representation of $G$ is $G'$-admissible. In particular, there exist an infinite dimensional irreducible unitary representation $\pi$ of $G$ and an irreducible unitary representation $\tau$ of $G'$ such that\[0<\dim_\mathbb{C}\mathrm{Hom}_{\mathfrak{g}',K'}(\tau_{K'},\pi_K|_{\mathfrak{g}'})<+\infty\]where $K$ is a maximal compact subgroup of $G$, $\mathfrak{g}$ is the complexified Lie algebra of $G$, $\pi_K|_{\mathfrak{g}'}$ is the restriction of underlying $(\mathfrak{g},K)$-module of $\pi$ to the complexified Lie algebra $\mathfrak{g}'$ of $G'$, $K'$ is a maximal compact subgroup of $G'$ such that $K'\subseteq G'\cap K$, and $\tau_{K'}$ is the underlying $(\mathfrak{g}',K')$-module of $\tau$.
\end{corollary}
\begin{proof}
The first conclusion follows from \corollaryref{17} and \theoremref{25} immediately, and the second statement follows from the correspondence between unitary representations of Lie groups and underlying $(\mathfrak{g},K)$-modules.
\end{proof}

\begin{center}
\textbf{Acknowledgement}
\end{center}
The author would like to express the thankfulness to professor Jun YU who gave some helpful suggestions and supports to this article.

\end{document}